\documentclass[authoryear,preprint,review,12pt]{elsarticle}

\usepackage{amssymb}
\usepackage{amsmath}
\usepackage{amsthm}

\usepackage{mathabx}
\usepackage{physics}

\usepackage[utf8]{inputenc}
\usepackage{a4wide}
\usepackage{color}
\usepackage{mathtools}
\newcommand{\R}{\mathbb{R}}
\newcommand{\M}[0]{{S}}

\newcommand{\PP}[0]{\mathcal{P}}
\newcommand{\Prob}[0]{\text{Prob}}
\newcommand{\diag}[0]{\text{diag}}

\newtheorem{theorem}[]{Theorem}
\newtheorem{corollary}[]{Corollary}
\newtheorem{lemma}[]{Lemma}
\newtheorem{example}[]{Example}
\newtheorem{definition}[]{Definition}

\journal{Systems and Control Letters}

\begin{document}

\begin{frontmatter}



\title{Lyapunov Densities For Markov Processes: An Application To Quantum Systems With Non-Demolition Measurements} 


\author[ozkan]{Özkan Karabacak\corref{ozkanc}} 
\cortext[ozkanc]{Corresponding author}
\affiliation[ozkan]{organization={Kadir Has University, Department of Mechatronics Engineering},
            addressline={Cibali Mah.}, 
            city={İstanbul},
            postcode={34083}, 
            country={Turkey}}

\author[horia]{Horia Cornean} 

\affiliation[horia]{organization={Aalborg University, Department of Mathematical Sciences},
            addressline={Skjernvej 4A}, 
            city={Aalborg},
            postcode={9220}, country={Denmark}}

\author[rafael]{Rafael Wisniewski} 

\affiliation[rafael]{organization={Aalborg University, Department of Electronic Systems},
            addressline={Fredrik Bajers Vej 7}, 
            city={Aalborg},
            postcode={9220},
            country={Denmark}}

\begin{abstract}
Stochastic convergence of discrete time Markov processes has been analysed based on a dual Lyapunov approach. Using some existing results on ergodic theory of Markov processes, it has been shown that existence of a properly subinvariant function (counterpart of the Lyapunov density in deterministic systems) implies sweeping of a Markov process out of the sets where this function is integrable. Such a function can be used as a certificate of convergence in probability of a stochastic system. We apply this technique to Markov processes induced by a quantum system with non-demolition measurement and propose dual Lyapunov certificates 
{to certify sweeping.}
\end{abstract}





\end{frontmatter}



\section{Introduction}
\label{sec:Intro}

Stability certificates are used for the analysis and control of both deterministic and stochastic dynamical systems. The existence of Lyapunov functions and dual Lyapunov functions (also called Lyapunov density) can imply certain asymptotic behaviour for deterministic systems. On the other hand, Lyapunov-type functions can also be employed for stochastic systems to certify certain types of stochastic convergence \citep{Kushner1967STOCHASTICCONTROL}. More specifically, Bielecki functions certify an asymptotic property called sweeping \citep{Lasota1994ChaosNoise} for Markov stochastic processes. On the other hand, the existence of a subinvariant, absolutely continuous measure 
implies the so-called Foguel alternative: either there exists an invariant density or the system is sweeping. In this work, we show that the existence a properly subinvariant absolutely continuous measure directly implies sweeping as it prohibits the existence of an invariant density. This result, which can be seen as a  dual version of the Bielecki function theorem, is a consequence of the literature on ergodic theory of Markov processes \citep{Foguel1969TheProcesses,Lasota1994ChaosNoise,Komorowski1989AsymptoticOperators} but is not stated and proven elsewhere.

Our work is inspired by an alternative approach to Lyapunov theory for deterministic systems, which is the so-called dual Lyapunov approach developed by \citep{Rantzer2001ATheorem} and has found applications in control of nonlinear autonomous systems \citep{Prajna2004NonlinearOptimization}, safety and reachibility analysis \citep{Prajna2007ConvexSystems,Kivilcim2019SafeDensity}, stability certification of switched systems \citep{Karabacak2020AlmostSwitching} and certification of synchronization of coupled oscillators \citep{Kudeyt2022CertificationOscillators}. Dual Lyapunov approach leverages measurable functions (called Rantzer's function or Lyapunov density) that are Radon-Nikodym derivatives of absolutely continuous subinvariant measures (called Lyapunov measure in \citep{Vaidya2010NonlinearMeasure}). It can be shown via standard techniques from potential theory \citep{Syski1973PotentialChains} that the existence of a properly subinvariant measure, which is necessarily infinite under the assumption that solutions exist globally in time, implies that almost all solutions converge to the set near which the measure is infinite \citep{Karabacak2018OnSets}. 

Applications of the dual Lyapunov theory to stochastic systems appeared for stochastic differential equations \citep{VanHandel2006AlmostStability} and for discrete-time maps with independent, identically distributed noise \citep{Vaidya2015StochasticMeasure}. However, up to our knowledge, dual Lyapunov approach has not been applied to general Markov processes, where dynamics is given by a linear operator (called Markov operator) on the set of Lebesgue integrable functions \citep{Foguel1969TheProcesses,Lasota1994ChaosNoise}. 

It has been shown in \citep{Foguel1966LimitProcesses} that a Markov process is sweeping if it has no invariant density but admits a subinvariant function. In \citep{Komorowski1989AsymptoticOperators}, various extensions of this result have been obtained. \citet{KOMORNIK1995THEOPERATORS} discuss that if the system is dissipative, then the condition on the existence of a subinvariant function is redundant; and they characterize the sweeping property based on the absence of an invariant density and the existence of a particular property, which implies that the system is dissipative and therefore a construction of a subinvariant function is possible. In this work, to come up with an easily checkable condition, we waive the condition on the absence of an invariant density (which might be nontrivial to show as in \cite[theorem~4.2]{Tyrcha1988AsymptoticCycle}) and show that the existence of a properly subinvariant function implies the sweeping property.

We apply our main result to a quantum system with non-demolition measurements. Non-demolition measurements are represented by invertible transformations providing information about the state without collapsing it onto certain subspaces. This type of measurements have been utilized to stabilize a photon box system via feedback control, leading to a Nobel prize \citep{Haroche2013NobelBoundary}. Stochastic convergence of this system has been analyzed in \citep{Amini2012StabilizationCase-study,Amini2013FeedbackDelays}, where stochastic Lyapunov functions \citep{Kushner1967STOCHASTICCONTROL} were employed. We show that the dual Lyapunov approach can also provide (not yet new) results on stochastic convergence for the photon box system.


\section{Sweeping Markov Operators}
In this section, 
we prove that the existence of a properly subinvariant function implies sweeping of the process from an admissible family of subsets on which the function has finite integral. 
As an example, we show that using such a subinvariant function sweeping can be shown more directly for the cell size distribution in a cell cycle processes \citep{Tyson1986MathematicalCycle, Tyrcha1988AsymptoticCycle,Komorowski1989AsymptoticOperators}.  
Finally, we apply our main result to iterative function systems where the probability of choice of the active function depends on the current state \citep{Barnsley1988InvariantProbabilities}. This will be used, in the next section, to prove stochastic convergence of trajectories for quantum systems involving non-demolition measurements.

\subsection{Preliminaries: Hopf Decomposition, Subinvariant Functions and Sweeping}
If not stated otherwise, we follow the setting used by \cite{Foguel1969TheProcesses,Lasota1994ChaosNoise,Komorowski1989AsymptoticOperators}, where all arguments hold up to zero measure sets. Let  $(X,\Sigma,m)$ be a $\sigma$-finite measure space. For a measurable function $f:X\to X$, $\|f\|$ denotes the $L_1$-norm of $f$, namely $\|f\|=\int |f| \ dm$. $L_1(X)$ and $L_\infty(X)$ denote the set of measurable functions with a finite $L_1$-norm and the set of bounded measurable functions, respectively, where two functions differing only on a measure zero set are assumed to be identical. A linear operator $\mathcal P:L_1(X)\to L_1(X)$ is a Markov operator if $\mathcal P$ is 
a contraction ($\|\PP f\|\leq \|f\|$) and positive ($f\geq 0 \implies \mathcal Pf\geq 0$). In this case, $(X,\Sigma,m,\mathcal P)$ is called a Markov process \citep{Foguel1969TheProcesses}. An equivalent probabilistic definition of a Markov process is given via transition probabilities $P(x,A)$ defined for all $x\in X$ and for all $A\in\Sigma$ such that the conditions
\begin{align}
&0\leq P(x,A)\leq 1,
\label{eq:TransitionProbability1}
\\
&P(\cdot,A) \text{ is measurable},
\label{eq:TransitionProbability2}
\\
&A=\bigcup_i A_i
\text{ for disjoint } A_i\text{'s} 
\implies P(x,A)=\sum_iP(x,A_i),
\label{eq:TransitionProbability3}
\\
&m(A)=0\implies P(x,A)=0 \text{ for almost all } x\in X
\label{eq:TransitionProbability4}
\end{align}
are satisfied for each $A\in\Sigma$.
The correspondence between these equivalent definitions are given via
\begin{equation}
P(x,A)=U1_A(x),
\end{equation}
where $1_A$ is the characteristic function of $A$ and $U:L_\infty(X)\to L_\infty(X)$ is the adjoint operator of $\PP$. Equivalently, the Markov operator $\PP$ satisfies
\begin{align}
    \int_A (\mathcal{P} \rho)(x) m(dx) &= \int_{X} (\mathcal{P} \rho)(x) 1_A(x) m(dx)\\ &= \int_X  \rho (x) P(x,A) m(dx) , 
\end{align}
for all $A \in \Sigma$,
where duality between $\mathcal P$ and $\mathcal U$ is used in the second equality.

\begin{definition}[Hopf decomposition \citep{Foguel1969TheProcesses}]
\label{def:Hopf}
Let $u\in L_1(X)$ and $u>0$. The set
$$C:=\left\{ x\in X:\,  \sum_{k=0}^{\infty}\big (\PP^ku\big )(x)=\infty\right\}$$
does not depend on $u$ and it is called the conservative part of $X$, while $D=X\setminus C$ is called the dissipative part. $\PP$ is said to be dissipative if $X=D$ (up to a set of measure zero).
\end{definition}
{
The following lemma characterizes the result of the infinite sum in Definition~\ref{def:Hopf}, when $u\geq 0$.
\begin{lemma}[{\citet[p.11]{Foguel1969TheProcesses}}]
\label{lem:FromFoguelBookp11}
If $u\in L_1(X)$ and $u\geq 0$ then
\begin{equation*}
    \sum_{k=0}^{\infty}\big (\PP^ku\big )(x)\in\begin{dcases}
    \{0,\infty\}&\text{for }x\in C\\
    [0,\infty)&\text{for }x\in D
    \end{dcases}
\end{equation*}
\end{lemma}
}
The operator $\PP$ can be extended to the set of all non-negative measurable functions as
$$\big (\PP f\big )(x)=\lim_{k \to \infty}(\PP f_k)(x),$$
where $\{f_k\}\subset L_1(X)$ is any sequence of non-negative measurable functions  pointwise converging monotonically to $f$ almost everywhere \cite[p.4]{Foguel1969TheProcesses}. Note that $\PP f$ may be equal to infinity on a set of positive measure. 
\begin{lemma}[{\citet[p.16]{Foguel1969TheProcesses}}]
\label{lem:FromFoguelBook}
Let $u$ be a real, nonnegative, measurable  function on $X$ such that $\PP u\leq u$ on $C$. Then $\PP u=u$ on $C$.
\end{lemma}
A measurable function $f:X\to X$ is called  
\begin{itemize}
    \item invariant  if $\PP f=f$
    \item subinvariant  if $\PP f\leq f$
    \item properly subinvariant  if $\PP f<f$ 
    \item a density if $f\geq 0$ and $\|f\|=1$
\end{itemize} 
In the following, we introduce the so-called sweeping property of Markov operators and some other definitions from \citep{Komorowski1989AsymptoticOperators}:
\begin{definition}[Sweeping operator]
A Markov operator $\PP $ is called sweeping (with respect to a familly $\mathcal A\subset\Sigma$) if 
\begin{equation}
    \label{eq:sweeping}
    \lim_{n\to \infty}\int_A \big ( \PP ^n f\big ) \, dm=0, \quad \forall  A\in\mathcal A\text{ and } f\in L_1(X).
\end{equation}
In particular, if $X\subset\mathbb R^N$ or $\mathbb C^N$ and the family $\mathcal A$ consists of countably many open sets $A_n, \ n=1,2,\dots$, then we say that $\PP $ is sweeping to 
\begin{equation}
\label{def:Lambda}
\Lambda:=\bigcap_{A\in\mathcal A}A^c
\end{equation}
if $\Lambda$ is non-empty; otherwise, we say that $\PP $ is sweeping to infinity\footnote{Here, we consider the standard metric in $\mathbb R^n$ or $\mathbb C^n$ and hence this argument and the definition of $\Lambda$ is not up to measure zero sets.}.
\end{definition}
\begin{definition}[Admissible family]
A family $\mathcal A\subset\Sigma$ is called admissible if 
\begin{itemize}
    \item $m(A)<\infty$ for $A\in\mathcal A$,
    \item $A_1,A_2\in\mathcal A$ implies $A_1\cup A_2\in \mathcal A$,
    \item {there exists a countable subfamily  $ \{A_n\}\subset \mathcal A$ such that $\bigcup_nA_n=X$.\footnote{We take this equality in the same form as in the literature where equality holds again up to measure zero sets. Hence, this equality does not imply that $\Lambda$ defined in \eqref{def:Lambda} is empty but it implies that $\Lambda$ has measure zero.}}   
\end{itemize}
\end{definition}
\begin{definition}[Locally integrable function]
   A measurable function $f:X\to \R_+$ is called locally integrable with respect to $\mathcal A$ if
$$\int_A f dm<\infty, \quad \text{for all }   A\in\mathcal A.$$ 
\end{definition}

\subsection{Main Result: Dual Lyapunov Function for Sweeping of Markov Processes}
We now show that existence of a properly subinvariant, positive function implies sweeping of the Markov process.  
The following directly follows from Lemma~\ref{lem:FromFoguelBook}:
\begin{lemma}
\label{lem:X=D}
If a properly subinvariant, positive function $u$ exists then $\PP $ is dissipative.
\end{lemma}
\begin{proof}
Recall that all the results in this section hold up to zero measure sets. Lemma~\ref{lem:FromFoguelBook} implies that $\PP u=u$ on $C$. Since $\PP u<u$ (almost everywhere), this implies $C$ is a zero measure set. In other words, $X=D$ (up to a zero measure set).
\end{proof}
\begin{definition}[Smoothing operator]
    A Markov operator $\PP $ is called smoothing (with respect to $\mathcal A$) if for any given $A\in\mathcal A$, $f\in L_1(X)$ and $\epsilon>0$ there is $\delta>0$ such that 
\begin{equation}
    \label{eq:smoothing}
    E\in\Sigma,\ E\subset A,\ m(E)\leq \delta \implies \int_E \big ( \PP ^n f\big ) \, dm\leq \epsilon, \quad \text{for all  } n=1,2,\dots
\end{equation} 
\end{definition}
The following is proven in \citep{Komorowski1989AsymptoticOperators}:
\begin{lemma}
\label{lem:smoothing}
    If a subinvariant, nonnegative function exists then $\PP $ is smoothing.
\end{lemma}
\begin{proof}
    See the proof of \cite[Corollary~2.2]{Komorowski1989AsymptoticOperators}
\end{proof}

A version of the so-called Foguel alternative given in \cite[Theorem~3.1]{Komorowski1989AsymptoticOperators} states that a Markov operator $\PP $ induced by a stochastic kernel is sweeping (with respect to an admissible family $\mathcal A$) if there is no invariant density but there exists a subinvariant, locally integrable (with respect to $\mathcal A$), positive function. To obtain a similar result for arbitrary Markov operators, we use the following relation between sweeping and dissipativeness:
\begin{theorem}[{\citet[Theorem 4.1]{Komorowski1989AsymptoticOperators}}]
    If a Markov operator $\PP $ is dissipative; i.e., $X=D$, and smoothing with respect to an admissible family $\mathcal A\subset\Sigma$ then $\PP $ is sweeping with respect to $\mathcal A$.
    \label{th:smoothing}
\end{theorem}
Based on the above definitions and results, we obtain the following characterizations of sweeping via properly subinvariant functions:
\begin{theorem}
    \label{th:sweeping}
  Let $(X,\Sigma,m)$ be a $\sigma$-finite measure space. A Markov operator $\PP :X\to X$ is sweeping with respect to an admissible family $\mathcal A\subset \Sigma$, if there exists a properly subinvariant, locally integrable (with respect to $\mathcal A$),  positive function $u$.   
\end{theorem}
\begin{proof}
By Lemma~\ref{lem:X=D}, {$\mathcal P$} is dissipative, and by Lemma~\ref{lem:smoothing}, it is also smoothing. Hence Theorem~\ref{th:smoothing} implies the result.
\end{proof}

Let us assume that $X=\mathbb R^N$ or $\mathbb C^N$, $\|\PP f\|=\|f\|$ for all $f\in L_1(X)$ and consider the implication of sweeping when  $\left(\bigcup_{A\in \mathcal A}A\right)^\text c$ is a singleton, say $\{\bar \phi\}$. For a given $f\in L_1(X)$ with $\|f\| = 1$, define probability measures $\mu_n(V):=\left(\int_V\PP ^nf\ dm\right)$. If $\PP $ is sweeping with respect to $\mathcal A$, then $\lim_{n\to\infty}\mu_n(A)=\lim_{n\to\infty}\int_A\PP^n f \ dm=0$
for any $A\in\mathcal A$. Hence, for a bounded continuous function $g$, we have $\int gd\mu_n\to g(\bar \phi)$. 
In other words, $\mu_n$ converges weakly to the Dirac measure $\delta_{\bar x}$. Let $\Phi_n$ be a random variable taking values on $X$ with respective distributions $\PP ^nf$. Then, by \cite[Remark 5.2.b]{Cinlar2011ProbabilityStochastics}, we can conclude that $\Phi_n$ converges to $\bar\Phi=\bar \phi$ in probability:

\begin{corollary}
\label{cor:convergenceprobability}
Let $(X=\mathbb R^N \text{ or }\mathbb C^N,\Sigma,m,\PP )$ be a Markov process, $\|\PP f\|=\|f\|$ for all $f\in L_1(X)$ and $\Phi_0$ be a random variable defined on $X$ and taking values on $X$ with a distribution that is absolutely continuous with respect to $m$ with Radon-Nikodym derivative being $\rho_0\in L_1(X)$. Let $\mathcal A\subset \Sigma$ be an admissible family such that $\bigcap_{\{A\in\mathcal A\}}A^\text c$ is a singleton, say $\{\bar\phi\}$. Then, the stochastic process $\{\Phi_n\}$ with a sequence of distributions $\{\PP^n\rho_0\}$ converges in probability to the random variable $\bar\Phi\equiv \bar\phi$ if there exists a positive, properly subinvariant, locally integrable (with respect to  $\mathcal A$) measurable function $u$. 
\end{corollary}

The following example illustrates an application of Theorem~\ref{th:sweeping} to a stochastic process modeling the cell size dynamics in a cell cycle process.


\begin{example}
We consider the dynamics of the cell size distribution proposed in \citep{Tyson1986MathematicalCycle} as
\begin{equation}
 \rho_{n+1}(x)=\PP \rho_n(x):=\int_\sigma^\infty K(x,y) \rho_n(y)\ dy,
\end{equation}
where $\rho_n$ is the distribution of cell sizes at birth for the $n$th generation cells, $\sigma$ is the minimum size at birth and $K(x,y)$ is a stochastic kernel given by
\begin{equation}
K(x,y)=\begin{dcases}
        \frac{\alpha}{\sigma} \left(\frac{x}{\sigma}\right)^{-1-\alpha} & , \sigma\leq y < 1\\
        \frac{\alpha}{\sigma} \left(\frac{x}{\sigma}\right)^{-1-\alpha}y^\alpha & , 1\leq y < \frac{x}{\sigma}\\
        0 & , \frac{x}{\sigma}\leq y.\\
    \end{dcases}
\end{equation}
\end{example}
Here $\alpha$ is a positive parameter. For $\alpha\ln \sigma > -1$, it is known that there is no invariant density \cite[Theorem~4.2]{Tyrcha1988AsymptoticCycle} and that $\bar\rho(x)=x^{-1+\beta}$ is invariant for some $\beta>0$ \citep{Komorowski1989AsymptoticOperators}. Note that $\bar\rho(x)$ is locally integrable with respect to the family of subsets $\{[\sigma,a)\mid a>\sigma\}$ which implies that the system is sweeping to infinity if $\alpha\ln \sigma\geq -1$  \citep{Komorowski1989AsymptoticOperators}. We now show that this can also be proven using directly Theorem~\ref{th:sweeping} without the need to show the absence of an invariant density 
\begin{align}
    &\PP \bar\rho=\PP (x^{-1+\beta})\nonumber\\ &=\frac{\alpha}{(\alpha+\beta)\sigma^{\beta}}\ x^{-1+\beta}+\alpha\frac{\alpha-\sigma^{\beta}(\alpha+\beta)}{\beta(\alpha+\beta)\sigma^{-\alpha}}\ x^{-1-\alpha}\label{eq:cellPerron}
\end{align}
Hence, $\PP (x^{-1+\beta})<x^{-1+\beta}$ if 
\begin{equation}
    \label{eq:cell}
    f(\beta):=\alpha-\sigma^\beta(\alpha+\beta)<0.
\end{equation}
Note that the first term in \eqref{eq:cellPerron} is less than $x^{-1+\beta}$ whereas the second term is negative.
$f(0)=0$ and $f'(0)=-\alpha \ln \sigma -1<0$ imply that for a sufficiently small positive $\beta$, \eqref{eq:cell} is satisfied and  $x^{-1+\beta}$ is properly subinvariant. Hence, Theorem~\ref{th:sweeping} implies that the system is sweeping to infinity.

\subsection{An Extension of Frobenius-Perron Operators}
A special class of Markov operators describing the dynamics of probabilities under a measurable, nonsingular\footnote{S is nonsingular if $m(S^{-1}(A))=0$ for all $A$ with $m(A)=0$}, deterministic map $S:X\to X$ are Frobenius-Perron operators \citep{Lasota1994ChaosNoise}. In this case, $P(x,A)=(U1_A)(x)=1_A(S(x))$ and the Markov operator $\PP $ can be obtained using the Radon-Nikodym derivative of the measure $\hat\mu_\rho(A)=\int_{S^{-1}(A)}\rho\ dm$ with repect to the Lebesgue measure $m$ as
\begin{equation}
\PP \rho=\frac{d\hat\mu_\rho}{dm}.
\end{equation}
Equivalently, $\mathcal P$ satisfies
\begin{align}
    \int_A (\mathcal{P} \rho)(x) m(dx) =  \int_{S^{-1}(A)}  \rho (x) m(dx)
\end{align}
If $S$ is differentiable and invertible then Frobenius-Perron operator has the explicit form \cite[Corollary~3.2.1 and Remark~3.2.4]{Lasota1994ChaosNoise}
\begin{equation}
\label{eq:invertibleS}
    \mathcal P\rho(x)=\rho(S^{-1}(x))\cdot |\det DS^{-1}(x)|,
\end{equation}
where $D$ stands for the Jacobian. 

Let us now consider an extension of the Frobenius-Perron operator to systems where  
the transition probabilities are given via
\begin{equation}
    \label{eq:NewTransitionProbability}  P(x,A)=\sum_{k=1}^Kp_k(x)\cdot 1_A(S_k(x)),
\end{equation}
where $\sum_{k=1}^Kp_k(x)=1$ for all $x\in X$ and $p_k(\cdot)$ is measurable for all $k$. Such processes arise from the iterated function systems where probability of choice of a function depends on the state \citep{Barnsley1988InvariantProbabilities}.
This induces a stochastic process $\{X_n\}$ {with transition probabilities}
\begin{align*}
   \text{Prob} [X_{n+1} \in A \mid  X_n = x ]&=\sum_{k=1}^K\text{Prob} [\kappa_n=k \mid X_n = x ]\cdot\text{Prob} [X_{n+1} \in A \mid \kappa_n= k, X_n = x ]\\
   &= \sum_{k=1}^K p_k(x)\cdot 1_A(S_k(x)), 
\end{align*}
where $\kappa_n$ is a discrete random variable (taking values in $\{1,\dots,K\}$) dependent on $X_n$ via the transition kernel $p_k(x)$.
The corresponding Markov operator $\PP $ satisfies


\begin{align*}
    \int_A (\mathcal{P} \rho)(x) m(dx) &=  \int_X  \rho (x) P(x,A) m(dx) \\
    &= \sum_{k=1}^K\int_{S_k^{-1}(A)} \rho(x)p_k(x)m(dx)=\sum_{k=1}^K\int_{X} (\mathcal P_k(\rho p_k))(x)m(dx),
\end{align*}
where $\mathcal P_k$ is the Frobenius-Perron operator for the deterministic map $\M_k$.
When $S_k$'s are differentiable and invertible, by \eqref{eq:invertibleS}, we get 
\begin{align}
\label{eq:invertibleSK}
    \mathcal P\rho(x)=\sum_{k=1}^Kp_k(S_k^{-1}(x))\rho(S_k^{-1}(x))\cdot |\det DS_k^{-1}(x)|,   
\end{align}

\section{Application To Quantum Systems With Non-Demolition  Measurements}
\label{sec:Quantum}

In this section, we show that Markov operators in the form of \eqref{eq:invertibleSK} arise in the analysis of quantum systems with non-demolition measurements and hence Lyapunov densities can be used to prove stochastic convergence of such systems.  

For a real or complex vector $\phi$, let $\|\phi\|$ denote its 2-norm. The set of all possible states of an $N$ dimensional quantum system is the unit complex sphere  
$$\mathcal{S}_N\subset \mathbb{C}^N,\quad \mathcal{S}_N=\{\phi\in\mathbb C^N:|\|\phi\|=1\}.$$
Consider a set of non-demolition measurement operators 
$\M_k:\mathcal{S}_N\to \mathcal{S}_N$, $k=1,\dots,K$ defined by 
\begin{equation}
\label{eq:NonDemolition}
\M_k(\phi):=\frac{M_k\phi}{\|M_k\phi\|} 
\end{equation}
such that $M_k$'s are positive definite, invertible matrices satisfying
\begin{equation}
\label{eq:completeness}
    \sum_{k=1}^K M_k^*M_k=I. 
\end{equation}
Note that 
\begin{equation}
\label{eq:NonDemolitionInverse}
\M_k^{-1}(\phi)=\frac{M_k^{-1}\phi}{\|M_k^{-1}\phi\|} .
\end{equation}
For a fixed time $n$, we introduce a random variable $\Phi_n$ taking values in $\mathcal{S}_N$ and a random variable $\kappa_n$ taking values in $\{1,\dots,K\}$ representing the outcome of the  measurement at time $n$. The probability of obtaining the outcome $\kappa_n=k$ is independent from $n$ and is given by 

\begin{equation}
p_k(\phi):=\Prob\left(\kappa_n=k|\Phi_n=\phi\right)=\|M_k\phi\|^2. 
\end{equation}
\eqref{eq:completeness} ensures the unity of the sum of probabilities. 
The following discrete-time Markov process arises:
\begin{equation}
\label{eq:Markov}
\Phi_{n+1}=\M_{\kappa_n}(\Phi_n)
\end{equation}

Note that 
$\mathcal{S}_N$
can be identified with the $2N-1$ dimensional real sphere $\mathbb{S}^{2N-1}$ through the map 
$$(x_1+iy_1,\dots, x_N+iy_N)\mapsto (x_1,y_1,\dots, x_N,y_N),\quad \sum_{j=1}^{N}(x_j^2+y_j^2)=1.$$
Let us consider the Riemannian measure on $\mathcal{S}_N$ induced by $\mathbb{S}^{2N-1}$ and let $L^1(\mathcal{S}_N)$ denote the set of Lebesgue integrable (with respect to the Riemannian measure) functions on $\mathcal{S}_N$. 
By \eqref{eq:invertibleSK}, the Frobenius-Perron operator $\PP:L^1(\mathcal{S}_N)\to L^1(\mathcal{S}_N)$ related to the dynamics 
driven by \eqref{eq:Markov} can be obtained as
\begin{align}
\label{eq:PerronOnTildeX}
\PP\rho(\phi)&=
\sum_{k=1}^Kp_k(\M_k^{-1}\phi)\cdot\rho(\M_k^{-1}\phi)\cdot \left|\det\left(D\M_k^{-1}(\phi)\right)\right|\\
&=\sum_{k=1}^K \frac{1}{\|M_k^{-1}\phi\|^2}\cdot\rho(\M_k^{-1}\phi)\cdot \left|\det\left(D\M_k^{-1}(\phi)\right)\right|.
\label{eq:PerronOnTildeX2}
\end{align}


 We first calculate the determinant of the Jacobian of an operator of form \eqref{eq:NonDemolition} on a real sphere:
\begin{lemma}
 \label{lemmahc1}Let 
 $\M:\mathbb{S}^{N-1}\to \mathbb{S}^{N-1}$ be defined by $\M(\phi):=M\phi/\|M\phi\|$, where $\phi\in\mathbb R^N$ with $\|\phi\|=1$ and $M$ is an invertible $N\times N$ real matrix. Then
 \begin{equation*}
 |\det\left(D\M(\phi)\right)|=\frac{|\det(M)|}{\|M\phi\|^N},
 \end{equation*}
where $DS(\phi)$ is the 
Jacobian acting on the tangent space at $\phi$ seen as an element of $\mathbb{S}^{N-1}$. 
 \end{lemma}

\begin{proof}

    At the outset, we compute the differential $DF$. 
    Let $\psi: (-\epsilon, \epsilon) \to \mathbb{S}^{N-1}$ be a path such that $\psi(0) = \phi$, and
$\dot \psi (0) = v$. Consequently, 
\begin{align*}
DS(\phi) v &= \left.\frac{d}{dt} \right|_{t=0} S(\psi(t)) = 
\frac{M v}{\|M \phi\|} - \frac{M \phi}{2 \|M \phi\|^3} \left( \phi^T M^{T} M v + v^T M^{T} M \phi \right) \\
&= \|M \phi\|^{-1} \left( M v - S(\phi)^T S(\phi) M  v \right) = \|M \phi\|^{-1} \left(\mathrm{Id}-   S(\phi)^T S(\phi) \right) M  v
\end{align*}
Hence,
\begin{align}
DS(\phi) =  \|M \phi\|^{-1} \left(\mathrm{Id}_N-  S(\phi) S(\phi)^T\right) M.
\label{DifferentialF}
\end{align}
    By direct computation, we conclude that $DF(\phi)\phi=0$; whereas its range is orthogonal to $S(\phi)$ when $\phi\in \mathbb{S}^{N-1}$, since $S(\phi)^T (\mathrm{Id}_N-  S(\phi) S(\phi)^T) = 0$. 

    Define $U \equiv U(\phi)$ to be an orthogonal $N\times N$ matrix which sends $S(\phi)$ into $\phi$, i.e., $US(\phi)=\phi$. If $\{f_j\}_{j=1}^{N-1}$ denotes an orthonormal basis of the subspace orthogonal to $S(\phi)$, then $e_j=Uf_j$ is an orthonormal basis of the subspace orthogonal to $\phi$.  Let us compute the matrix of $UM$ in the basis $\{\phi,e_1,...,e_{N-1}\}$. We have 
    $$a_{00}:=\phi^T\,UM\phi=S(\phi)^TM\phi=\|M\phi\|,$$
    $$a_{j0}:=e_j^T UM\phi=f_j^TM\phi = \|M \phi\| f_j^T S(\phi) = 0,\quad 1\leq j\leq N-1, $$
   and by \eqref{DifferentialF}
    $$a_{jk}=e_j^T UMe_k=f_j^T Me_k =\|M\phi\| \, [DS(\phi)]_{jk},\quad 1\leq j,k\leq N-1.$$ 
In the above, we have used $\left(\mathrm{Id}_N-  \ket{S(\phi)} \bra{S(\phi)}\right) f_j = f_j$.
    
    Thus, in this basis $\{\phi,e_1,...,e_{N-1}\}$, the matrix  $\|M\phi\|^{-1}\, UM$ looks like 
    $$\begin{bmatrix}
1& * & \dots & * \\
0 & [DS(\phi)]_{11} &\dots &[DS(\phi)]_{1,N-1}\\
  . & .&.&. \\
  . & .&.&. \\
  0& [DS(\phi)]_{N-1,1} &\dots &[DS(\phi)]_{N-1,N-1}
    \end{bmatrix}$$
where ``$*$" corresponds to the entries which are irrelevant for the further computation.   
    This implies that the determinant of $DS(\phi)$ 
    equals the determinant of $\|M\phi\|^{-1}\, UM$. Hence,
    $|\det (DS(\phi))| = |\det(\|M\phi\|^{-1}\, M)| = \frac{1}{\|M\phi\|^N} |\det(M)|$.
\end{proof}


The determinant of the Jacobian of an operator of form \eqref{eq:NonDemolition} on a unit complex sphere is as follows:
\begin{lemma}
 \label{lemmahc2}Let 
 $\M:\mathcal{S}_N\to \mathcal{S}_N$ be defined by $\M(\phi):=M\phi/\|M\phi\|$, where $\phi\in\mathbb C^N$ with $\|\phi\|=1$ and $M$ is an invertible matrix with complex entries. Then we have
 \begin{equation*}
 |\det\left(D\M(\phi)\right)|=\frac{|\det(M)|^2}{\|M\phi\|^{2N}},
 \end{equation*}
 where 
 $DS(\phi)$ is the Jacobian  acting on the tangent space at $\phi$ seen as an element of $\mathbb{S}^{2N-1}$.  
 \end{lemma}
 \begin{proof}
     Let $\phi=u+iv\in \mathcal{S}_N$ with $u,v\in \mathbb{R}^N$ such that $\|\phi\|^2=\|u\|^2+\|v\|^2=1$. We write $M = M_R + i M_I$. As a consequence, the real and complex part of $M \psi$ is
     $$
     \tilde M \begin{bmatrix}
         u \\ v
     \end{bmatrix}, \hbox{ with }  \tilde M = \begin{bmatrix}
        M_R & -M_I \\ M_I & M_R
    \end{bmatrix}.
     $$
     
     We observe that $\|M\phi\|=\|\tilde{M}[u^T,v^T]^T\|_{\mathbb{R}^{2N}}$ and compute the determinant of the block matrix $\tilde M$,
     \begin{align*}
         \det(\tilde M) = \det(M_R + i M_I) \det(M_R - i M_I) = \det(M) \overline{\det(M)} = |\det(M)|^2.
     \end{align*}
     Now we can apply Lemma \ref{lemmahc1}, where $M$ is replaced by $\tilde{M}$, and $N$ by $2N$. 
 \end{proof}

We can now obtain the Frobenius-Perron operator of a stochastic quantum system driven by quantum nondemolition measurements in terms of the measurement operators. The following directly follows from \eqref{eq:NonDemolitionInverse}, \eqref{eq:PerronOnTildeX2} and Lemma~\ref{lemmahc2}.
\begin{theorem}
\label{th:QNDPerron}
The Frobenius-Perron operator of
\eqref{eq:Markov} is given by
\end{theorem}
\begin{equation}
\label{eq:PerronSimple}
\PP\rho(\phi)=\sum_{k=1}^K\frac{|\det(M_k ^{-1})|^{ 2}}{\|M_k^{-1}\phi\|^{{ 2}N+2}}\cdot\ \rho\left(
\frac{M_k^{-1}\phi}{\|M_k^{-1}\phi\|}
\right).    
\end{equation}

In the following example, we apply Theorem~\ref{th:sweeping} to the Markov operator in Theorem~\ref{th:QNDPerron}:

\begin{example}
\label{ex:QND}
We consider the quantum system  studied by \cite{Amini2012StabilizationCase-study}, namely the system \eqref{eq:Markov} where
$M_k=\diag(m_k(1),\dots,m_k(N))$, $m_{k}(i) \in (0,1)$, $\sum_{k=1}^Km_k(i)^2=1
$ and $i_1\neq i_2$ implies that $m_k(i_1)\neq m_k(i_2)$ for all $k=1,\dots,K$ and $i=1,\dots,N$. A Lyapunov function has been found for this system in \citep{Amini2012StabilizationCase-study} showing almost sure convergence to the set of elementary vectors (corresponding to Fock states) for almost all initial quantum states. Here, to apply the dual Lyapunov method, we choose
\begin{equation}
    \label{eq:Density}
\rho(\phi)=\frac{1}{\prod_{i=1}^N{|\phi_i|^{ 2}}}.
\end{equation}
Then, applying Theorem~\ref{th:QNDPerron}, we get
\begin{eqnarray*}
    \label{eq:PerronDensity}
    \PP\rho(\phi)&=&\sum_{k=1}^K\frac{(\det(M_k^{-1}))^{2}}{\|M_k^{-1}\phi\|^{{2}N+2}}\cdot\rho\left(\frac{M_k^{-1}\phi}{\|M_k^{-1}\phi\|}
\right)\\
    &=&\sum_{k=1}^K\frac{\prod_{i=1}^N{\left|\frac{1}{m_k(i)}\right|^{2}}}{\|M_k^{-1}\phi\|^{{2}N+2}}\cdot\frac{1}{\prod_{i=1}^N{\left|\frac{{\phi_i}/{m_k(i)}}{\|M_k^{-1}\phi\|}\right|^{2}}}\\
    &=&\left(\sum_{k=1}^K\frac{1}{\|M_k^{-1}\phi\|^{2}}\right)\cdot\frac{1}{\prod_{i=1}^N{|\phi_i|^{2}}}\\
    &<&\frac{1}{\prod_{i=1}^N{|\phi_i|^{2}}}=\rho(\phi),
\end{eqnarray*}
where the last line is proven in \ref{app1}. By Theorem~\ref{th:sweeping}, the system \eqref{eq:Markov} is sweeping to $\bigcup_i\{\phi_i=0\}$. 
  We conjecture that, recursively applying the same argument for each subspace of type $\phi_i=0$, one may reach to the conclusion that $\phi$ converges in probability to the set of Fock states, which correspond to computational basis in $\mathcal{S}_N$. However, we could not rigorously prove this conjecture. 
\end{example}

\section{Conclusion}
The dual Lyapunov approach has been extended to a Markov process defined by a Markov operator and a sufficient condition for the sweeping property of a Markov process has been obtained. The result has been applied to the dynamics of cell size distributions in a cell cycle process and to the dynamics of a quantum system driven by nondemolition measurements. 

Unlike the standard Lyapunov techniques, such as Kushner's invariance theorem (see Theorem~6 in \citep{Amini2012StabilizationCase-study}), the dual method does not require a calculation of the largest invariant set. Conversely, demonstrating the existence of a Lyapunov density, particularly when the nonintegrability condition is limited to the vicinity of the stochastic attractor, can present significant challenges.
Practical tools for certifying the asymptotic behavior of Markov processes may be developed through computational methods designed to construct Lyapunov densities.


\section*{Acknowledgements}
ÖK acknowledges support from Scientific and Technological Research Council of Turkey, grant TÜBİTAK-1001 122E522. RW acknowledges support from Independent Research
Fund Denmark–Natural Sciences, for the profect entitled \emph{Q-Safe: Quantum-Probabilistic Safe Decision Algorithms }. HC acknowledges support from Independent Research
Fund Denmark–Natural Sciences, grant DFF–10.46540/2032-00005B.

\appendix
\section{A Lemma for the Proof of $\mathcal P\rho_1<\rho_1$ in Example~\ref{ex:QND}}
\label{app1}
\begin{lemma}
Let $M_k=\text{diag}(m_k(1),\dots,m_k(N))$ for $k=1,\dots,K$ such that $m_k(i)\in(0,1)$ and $\sum_{k=1}^Km_k(i)^2=1$. Then, for any nonzero $\phi\in\mathbb C^N$,
$$\sum_{k=1}^{K}\frac{\|\phi\|^{2}}{\|M_k^{-1}\phi\|^{2}}\leq 1,$$
where equality is satisfied only when for each $k$, $m_k(s)$ is constant for all $s$ satisfying $\phi_s\neq 0$.
\end{lemma}
\begin{proof}
Note that multiplying $\phi$ by a scalar does not change the above expression. Hence, without loss of generality, we assume that $\|\phi\|=1$.
By Cauchy-Schwarz inequality, we have
\begin{equation}
\label{eq:CauchySchwarz}
    1=\left(\sum_i|\phi_i|^2\right)^2\leq \left(\sum_i\left(\frac{|\phi_i|}{m_k(i)}\right)^2\right)\cdot\left(\sum_i\left(m_k(i)|\phi_i|\right)^2\right),
\end{equation}
where equality takes place only when $m_k(s)^2$ is constant for indices $s$ satisfying $\phi_s\neq0$ (this condition is automatically satisfied when $\phi_s=0$ for all $s$ but one). 

\eqref{eq:CauchySchwarz} implies that

\begin{equation}\label{dc20}
    \frac{1}{\sum_i\frac{|\phi_i|^2}{m_k(i)^2}}\leq\sum_im_k(i)^2|\phi_i|^2,
\end{equation}
Hence, we have
\begin{equation*}
    \sum_k\frac{1}{\sum_i\frac{|\phi_i|^2}{m_k(i)^2}}\leq\sum_k\sum_im_k(i)^2|\phi_k|^2=\sum_i\sum_km_k(i)^2|\phi_i|^2=\sum_i|\phi_i|^2=1,
\end{equation*}
where equality takes place if, for each $k$, $m_k(s)^2$ is constant for indices $s$ satisfying $\phi_s\neq0$.

\end{proof}

\bibliographystyle{elsarticle-harv} 
\bibliography{references1.bib}{}


\end{document}